\documentclass[a4paper,11pt]{amsart}
\usepackage{amsmath}\usepackage{amsthm}
\usepackage{latexsym}
\usepackage[psamsfonts]{amssymb}
\usepackage{mathrsfs}
\usepackage{graphicx}\usepackage{epstopdf}
\usepackage[colorlinks=true,linkcolor=black,citecolor=black]{hyperref}
\usepackage{enumitem}
\newtheorem{thm}{Theorem} 
\newtheorem{theorem}{Theorem}[section]

\newtheorem{lemma}[theorem]{Lemma}
\newtheorem{cor}{Corollary}
\newtheorem{assumpt}[theorem]{Assumption}
\newtheorem*{lemma*}{Lemma}
\theoremstyle{definition}
\newtheorem{hypo}[theorem]{Assumption}
\newtheorem*{hypo*}{Assumption}
\newtheorem{remark}[theorem]{Remark}
\newtheorem*{example*}{Example}
\newtheorem*{notation*}{Notation}
\newtheorem*{remark*}{Remark}

\numberwithin{equation}{section}

\newcommand{\comment}[1]{}
\newcommand{\la}{\,\langle\,}
\newcommand{\ra}{\,\rangle\,}

\newcommand{\ov}[1]{\overline{#1}}
\newcommand{\sslash}{\mathbin{\backslash\mkern-6mu\backslash}}

  \def\tr{\operatorname{Tr}}
 \def\stab{\operatorname{Stab }}
\def\sgn{\operatorname{sgn}}  

\def\PP{\mathbb{P}} \def\Q{\mathbb{Q}}\def\R{\mathbb{R}}\def\Z{\mathbb{Z}}\def\C{\mathbb{C}}
\def\le{\leqslant} \def\ge{\geqslant}
\def\SL{\mathrm{SL}} \def\GL{\mathrm{GL}}
 \def\H{\mathcal{H}}

\def\Sc{\mathcal{S}}  \def\M{\mathcal{M}}
\def\X{X}
  \def\cc{\mathcal{C}}
\def\e{\varepsilon} \def\DD{\Delta} \def\G{\Gamma}\def\om{\omega}\def\SS{\Sigma}
\def\dd{\delta} \def\ss{\sigma}
\def\wc{\widetilde{\chi}}
\def\a{\alpha}\def\g{\gamma}\def\aa{\mathfrak{a}}\def\bb{\mathfrak{b}}
\def\cps{\operatorname{Cusps}}
\def\+{\,+\,}   

\def\sm#1#2#3#4{\left(\begin{smallmatrix}#1&#2 \\ #3 & #4 \end{smallmatrix}\right)}
\def\pmat#1#2#3#4{\begin{pmatrix}#1&#2 \\ #3 & #4 \end{pmatrix}}
\def\be{\begin{equation}}  \def\ee{\end{equation}}

\def\vp{\varphi}\def\wg{\widetilde{\Gamma}}

\def\rar{\rightarrow}
\title[Trace formula for Hecke operators]{On the trace formula for 
Hecke operators\\ on congruence subgroups, II}

\author{Alexandru A. Popa}

\address{Institute of Mathematics ``Simion Stoilow" of the Romanian Academy,
P.O. Box 1-764, RO-014700 Bucharest, Romania}
\address{E-mail: aapopa@gmail.com}

\begin{document}
\maketitle

\begin{abstract}In a previous paper, we obtained a general trace formula for
double coset operators acting on modular forms for congruence subgroups, 
expressed as a sum over conjugacy classes. Here we specialize it to 
the congruence subgroups 
$\Gamma_0(N)$ and $\Gamma_1(N)$, obtaining explicit formulas in terms of class 
numbers for 
the trace of a composition of Hecke and Atkin-Lehner operators. The formulas
are among the simplest in the literature, and  hold without any 
restriction on the index of the operators. We give two applications of the
trace formula for $\Gamma_1(N)$: we determine explicit trace forms for 
$\Gamma_0(4)$ with Nebentypus, and we compute the limit of the trace of a 
fixed Hecke operator as the level $N$ tends to infinity.
\end{abstract}


\section{Introduction}\label{sec1}

In~\cite{P}, we gave a short proof of a trace formula for Hecke operators 
on modular forms of weight $k\ge 2$ for finite index subgroups of $\SL_2(\Z)$. 
The formula expresses the following combination of traces 
 \be\label{0} \tr([\SS], M_k(\G,\chi)+S_k(\G,\chi))\, := \, 
  \tr([\SS], M_k(\G,\chi))\+\tr([\SS], S_k(\G,\chi))\;, \ee
as a sum over conjugacy classes of simple factors (the notation 
and the precise formula is reviewed in Section~\ref{sec2}). In this paper
we specialize $\G$ to be one of the congruence subgroups~$\G_0(N)$ or~$\G_1(N)$,
and we obtain explicit trace formulas in terms of class numbers. These formulas 
are among the simplest 
in the extensive literature, requiring one extra ingredient compared to
the $\SL_2(\Z)$ case: an arithmetic function which is multiplicative in 
the level $N$. For example, for $\G=\G_0(N)$ and for~$\chi$ any character modulo $N$,
we obtain for all $k\ge 2$, $N\ge 1$ and $n\ge 1$:
\[\begin{split}
\tr(T_n,M_k(\G,\chi)+S_k(\G,\chi))\,=\,
-\sum_{t\in \Z}p_{k-2}(t,n)\cdot
\sum_{\substack{u|N\\u^2|t^2-4n}} H\Big(\frac{4n-t^2}{u^2}\Big) C_{N,\chi}(u,t,n) +\\
+\delta_{k,2}\delta_{\chi,{\bf 1}}\sum_{d|n,(d,N)=1} n/d
\end{split}
\]
where $p_{k-2}(t,n)$ is the Gegenbauer polynomial, $H(D)$ is the Kronecker-Hurwitz 
class number, extended to negative squares by $H(-u^2)=u/2$ for $u>0$ and to 
negative non-squares by 0, and $C_{N,\chi}(u,t,n)$ is a multiplicative function 
in the level $N$, defined in~\eqref{16}.

It was Zagier's insight that the trace formula for $\SL_2(\Z)$ 
takes the simplest form when written for the linear combination~\eqref{0}, 
and he expressed it in terms of the above extension of the Kronecker-Hurwitz class 
numbers to all integers~\cite{Z1,Z}. Here we show, based on the formula proved
in~\cite{P}, that the same phenomenon happens
for any congruence subgroup satisfying a mild assumption (Assumption~\ref{h}), 
which is verified by the usual congruence subgroups. 

While the combination of traces~\eqref{0} yields the simplest formula, 
for applications, and to compare it with the existing literature, it is useful 
to have a formula on the cuspidal subspace alone. The second goal of 
this paper is to extract the cuspidal trace from~\eqref{0}, by computing the 
trace on the Eisenstein subspace. We perform this computation for an arbitrary 
Fuchsian group of the first kind with cups, since the formula may be of independent 
interest, and since we expect the resulting formula for the cuspidal 
trace to hold for such a Fuchsian group as well.

The trace formula for Hecke operators on spaces of modular forms for 
Fuchsian groups has a long history, starting with the celebrated papers 
of Eichler~\cite{E56} and Selberg~\cite{Se}. The existing approaches 
lead to trace formulas on the cuspidal subspace alone, and the 
different types of conjugacy classes need separate 
treatments~\cite{E56,H74,Sh1,Sa,O77,Z1}, whereas in our approach all conjugacy classes
are treated uniformly through a simple invariant. 
Existing formulas are quite complicated, and the index of the Hecke operators 
is usually assumed coprime with the level. Among the simplest and most general
is the trace formula for $\G_0(N)$ with Nebentypus proved by Oesterl\'e by 
analytic means~\cite{C77, O77}, an equivalent form of which we recover in 
Theorem~\ref{T3}. 

We also obtain new formulas for the trace of a composition of 
Hecke and Atkin-Lehner operators for~$\G_0(N)$, both in terms of the regular class numbers, and in terms 
of  the Kronecker-Hurwitz class numbers, and we impose no restrictions on 
the index of the operators involved. A previous such formula was given by Skoruppa
and Zagier~\cite{SZ}, assuming the index of the Hecke operator coprime with
the level. We also include a trace formula for $\G_1(N)$, 
which turns out to be simpler than for $\G_0(N)$, and which is the first such 
formula available in the literature. 

We include two applications that illustrate the versatility of our  
trace formulas. Specializing the trace formula for~$\G_1(N)$ to $N=4$ 
we obtain that 
\[\sum_{n\ \mathrm{even}}\left( \sum_{\substack{t^2\le 4n\\ 4|t-n-1}}
p_{k-2}(t,n) H(4n-t^2)+\sum_{\substack{n=ad\\ a\ \mathrm{odd}}}
\min(a,d)^{k-1}
-\delta_{k,2}\sum_{\substack{n=ad\\ a\ \mathrm{odd}}} d\right) q^n
\]
is a cusp form in $S_k(\G_0(4))$ if $k\ge 2$ is even. The coefficients 
of this modular form are the traces $-\tr(T_n, S_k(\G_0(4))$ for $n$ even, 
and we obtain also an explicit generating series for the traces of $T_n$ 
with $n$ odd (see Corollary~\ref{C1} in Section~\ref{sec1.4}). While the 
explicit cusp form above seems not to have appeared before, the explicit generating series we get for odd coefficients 
was conjectured to be a cusp
form by H.~Cohen more than 30 years ago~\cite{Co}. Although Cohen suggested
that this modular form was related to the ``trace form'', whose 
coefficients are the traces of Hecke operators, the conjecture was only recently 
proved by different methods by M. H. Mertens~\cite{Me1}. In Corollary~\ref{C1} 
we also compute 
explicitly the trace form for $S_k(\G_0(4), \chi)$ if $k$ is odd and 
$\chi$ is the nontrivial character modulo 4. For $k=2$ or $k=3$, the 
corresponding spaces of cusp forms are trivial, and these formulas reduce 
to class number relations similar to the Kronecker-Hurwitz formula. 

As another application, we take the limit in the trace
formula for $\G_1(N)$ as $N\rar\infty$. We find 
that the trace of a fixed Hecke operator $T_n$ stabilizes, and the result is 
suprisingly independent of $n>1$ and $k\ge 2$:
\[
\lim_{\substack{N\rar \infty\\ (N,n-1)=1}} 
\frac{\tr(T_n, S_k(\G_1(N)))}{\varphi(N)}=-\frac{1}{2}
\]
(in fact for $N>2n+2$ coprime to $n-1$ and $k>2$ we have equality).
Such limit forms of the trace formula were used by Serre to prove equidistribution 
results for Hecke eigenvalues of a fixed Hecke operator on~$\G_0(N)$, as the level and 
weight tend to infinity~\cite{S}. Such limits also appear naturally in work of F. 
R\u{a}dulescu~\cite{R}, in the context of certain distributions on $\SL_2(\Z)$. 

The paper is organized as follows. In Section~\ref{sec2} we state the main results:
after recalling the main result of~\cite{P} in Theorem~\ref{T2}, the general
trace formula on the cuspidal space is stated in Theorem~\ref{TSS}. In \S \ref{sec1.3}
we show how  to derive from it explicit versions in terms of class numbers, 
under an additional assumption. The group $\G_0(N)$ satisfies
this assumption, yielding the explicit formulas in Theorems~\ref{T3} and~\ref{T4}. 
The trace formula for $\G_1(N)$ is given in Theorem~\ref{T5}, 
and its applications in Section~\ref{sec1.4}. 

In Section~\ref{secEis} we prove a trace formula on the Eisenstein subspace for 
a general Fuchsian group of the first kind (Theorem~\ref{T5.1}), and use it 
to prove Theorem~\ref{TSS}. Finally, in Section~\ref{sec5} we compute the 
arithmetic functions entering Theorem~\ref{TSS} in the case of 
$\G_0(N)$, thus proving Theorems~\ref{T3} and~\ref{T4}. 

\noindent {\bf Acknowledgements.}
This work was partly supported by the the European Community grant 
PIRG05-GA-2009-248569 and by the CNCS grant TE-2014-4-2077. Part of this work 
was completed during several visits at MPIM in Bonn, whose support 
I gratefully acknowledge. 

\section{Statement of results}\label{sec2}
Let $\G$ be a finite index subgroup of~$\G_1=\SL_2(\Z)$ and let $\SS$ be a double coset contained in the commensurator 
$\wg\subset\GL_2^+(\R)$. 
If $\chi$ is a character of $\G$ with kernel of 
finite index, the action of the double coset operator $[\SS]$ on $M_k(\G,\chi)$ 
is defined using a multiplicative function $\wc$ on the semigroup generated 
by~$\G$ and~$\SS$ inside~$\wg$, such that $\wc|_{\G}=\chi^{-1}$, namely
 \be\label{chi} \wc(\g\ss\g')=\chi^{-1}(\g\g')\wc(\ss)\;, \quad \text{ for all } \g\in\G, \ss\in\SS \;. \ee  
A modular form $f\in M_k(\G,\chi)$ satisfies $f|_k \g=\chi(\g) f$, and the double 
coset operator $[\SS]$ acts by
  \be \label{hecke}
  f|[\SS]=\sum_{\ss\in\G\backslash\SS} \det\ss^{k-1} \cdot \wc(\ss)\cdot f|_k \ss \;, \ee
where $f|_k\g (z)=f(\g z) (c_\g z+d_\g)^{-k}$, and we write $\g=\sm {a_\g}{b_\g}{c_\g}{d_\g}$ throughout the paper. 
\begin{example*} Let $\G$ be the congruence subgroup $\G_0(N):=\{\g\in\G_1 : N|c_\g\}$, and
let $\chi$ be a character modulo $N$ viewed as a character of 
$\G_0(N)$ by $\chi(\g)=\chi(d_\g)$. The usual Hecke operators $T_n$ on $M_k(\G,\chi)$ 
are associated to the double coset
  \be \label{delta}  \DD_n:=\{\ss\in M_2(\Z) \;:\; \det \ss=n,\ N|c_\ss,\ (a_\ss,N)=1\}\;, \ee
and $\wc(\ss)=\chi(a_\ss)$ for $\ss\in\DD_n$.   
\end{example*}
\subsection{A general trace formula on the cuspidal subspace}
For any subset $\Sc$ of $\GL_2^+(\R)$, we denote by
$\ov{\Sc}=(\Sc\cup -\Sc) /\{\pm 1\}\subset \GL_2^+(\R)/\{\pm 1\}$.
By scaling $\SS$ we may assume that it is contained in the set $\M$ of 
integral matrices of positive determinant. For a $\ov{\G}_1$ conjugacy class
$X\subset \ov{\M}$, we choose any representative $M_X\in \M$ and 
we denote by $\DD(\X)=\tr(M_\X)^2-4\det(M_\X)$ the 
discriminant of the quadratic form associated to $M_\X$, and by 
$|\stab_{\ov{\G}_1} M_\X|$ the (possibly infinite) cardinality of the stabilizer 
of~$M_\X$ under conjugation by~$\ov{\G}_1$. We define the conjugacy class 
invariant 
 \be\label{1.70} \e(\X)=\begin{cases} \phantom{xx}
\dfrac 16 &  \text{if $M_X$ scalar,}\vspace{2mm} \\
\dfrac{\sgn \DD(\X)}{|\stab_{\ov{\G}_1} M_\X|} &  \text{ otherwise.}           
            \end{cases} 
 \ee
Explicitly $\e(X)$ is equal to: 
1/6 if $M_X$ is scalar;
$-1/|\stab_{\ov{\G}_1} M_X|$ if $M_X$ is elliptic; 1 if $M_X$ is hyperbolic 
fixing two cusps of~$\G_1$; and 0 otherwise. 

To state the main result of~\cite{P}, let $p_w(t,n)$ be the Gegenbauer 
polynomial, defined by the power series expansion 
  $\ (1-tx+nx^2)^{-1}= \sum_{w\ge 0} p_w(t,n)x^w\;. $
\setcounter{thm}{0}
\begin{thm}[\cite{P}]\label{T2}Let $\G$ be a finite index subgroup of $\G_1$, $k\ge 2$ an integer, 
$\chi$ a character of~$\G$ with kernel of finite index in~$\G$, 
and~$\SS$ a double coset of $\G$ such that $|\G\backslash\SS|=|\G_1\backslash\G_1\SS|$. 
Assuming $\chi(-1)=(-1)^k$ if $-1\in\G$, we have 
 \be \label{TF2} 
 \begin{split}
 \tr([\SS], M_k(\G,\chi)+ S_k^c(\G,\chi)) \,=\,
 \sum_{\X}p_{k-2}(\tr M_\X, \det M_\X)\;\cc_{\G,\SS}^\chi(M_\X)\; \e(\X) \\
 \+\delta_{k,2}\delta_{\chi,{\bf 1}}\; \sum_{\ss\in\G\backslash\SS} \wc(\ss)\;,
 \end{split}
 \ee
where the sum is over $\ov{\G}_1$-conjugacy classes $\X\subset \ov{\G_1\SS\G_1}$ with representative $M_\X\in \G_1\SS\G_1$, 
and\footnote{The same sign is chosen in all three places in~\eqref{1.10}. 
If $-1\notin\G$ at most one choice of signs is possible for each $A$, while 
if $-1\in \G$ both choices yield the same value for the summand.}
 \be\label{1.10} \cc_{\G,\SS}^\chi(M)\; :=\; \sum_{\substack{A\in \ov{\G}\backslash\ov{\G}_1\\ 
 \pm AMA^{-1}\in \SS}}(\pm 1)^k\wc(\pm AMA^{-1}) \;.\ee
The symbol $\dd_{a,b}$ is 1 if $a=b$ and 0 otherwise. 
 \end{thm}
Here $S_k^c(\G,\chi))$ is the space of anti-holomorphic cusp forms. 
From now on we assume for simplicity that $\G$, $\SS$ and $\chi$ are invariant under 
conjugation by an order 2 element of determinant~$-1$ (an assumption 
satisfied for most congruence subgroups, in particula for $\G_0(N)$), so that 
we can replace the space~$S_k^c(\G,\chi)$ by $S_k(\G,\chi)$ in the theorem 
(see \cite[Remark~3.2]{P}). 

To state the formula on the cuspidal subspace, 
let $\G_\aa\subset \G$ and $\SS_\aa\subset \SS$ be the stabilizers of a cusp $\aa$ of $\G$. 
  Let $C(\G)$ be a set of representatives 
for the $\G$-equivalence classes of cusps and define the subset
  \[ C(\G,\chi)= \{\aa\in C(\G)\;:\; \chi(\g)=\sgn(\g)^k \text{ if $\g\in\G_\aa$} \}\;, \]
where $\sgn(\ss):=\sgn(\tr\ss)$ denotes the sign of the eigenvalues of 
a parabolic or hyperbolic matrix $\ss\in\GL_2^+(\R)$. 
\begin{thm}\label{TSS}With the notations of Theorem~\ref{T2} we have  
 \be\label{TFS2} \begin{split}
 \tr([\SS], S_k(\G,\chi))=
 \frac{1}{2}\sum_{\substack{\X,\, \DD(\X)\le 0}}p_{k-2}(\tr M_\X, \det M_\X) 
 \;\cc_{\G,\SS}^\chi(M_\X)\; \e(\X) \\
 - \frac{1}{2}\sum_{\aa\in C(\G,\chi)}\sum_{ \ss\in\G_\aa\backslash \SS_\aa/\G_\aa }
  \frac{ \min(|\lambda_\ss|,|\lambda_\ss'| )^{k-1} }{(|\lambda_\ss|,|\lambda_\ss'|) }  \sgn(\ss)^k \wc(\ss)
  +\delta_{k,2}\delta_{\chi,{\bf 1}}\; \sum_{\ss\in\G\backslash\SS} \wc(\ss)\;,
 \end{split} \ee
where $\lambda_\ss, \lambda_\ss'$ are the eigenvalues of $\ss$, and for $a,d\in\R$ with $a/d\in\Q$ we let
$(a,d)$ be the unique positive number such that $\Z a+\Z d=\Z(a,d)$.\footnote{We show in \eqref{5.4} that
$\lambda_\ss/\lambda_\ss' \in \Q$ for $\ss\in\SS_\aa$.} 
 \end{thm}
\begin{remark*}
The sum over cusps in \eqref{TFS2} can be written more explicitly as follows. 
Fix a scaling matrix $C_\aa$ for the cusp $\aa$, namely $C_\aa \aa=\infty$ and 
  \[\{\pm 1\}\cdot C_\aa \G_\aa C_\aa^{-1} =\{\pm \sm 1n01 \;:\; n\in\Z\}.\] 
For $a,d>0$ we set $\SS_\aa(a,d):=\{\ss\in\SS_\aa \;: \; C_\aa\ss C_\aa^{-1} 
=\sgn(\ss) \sm a*0d \}$  and 
\be\label{1.8} \Phi_{\G,\SS}^\chi(a,d):=\frac{1}{(a,d)} \sum_{\aa\in C(\G,\chi)}
  \sum_{ \ss\in\G_\aa\backslash \SS_\aa(a,d)/\G_\aa }\sgn(\ss)^k \wc(\ss) \;. 
\ee
 We show in Theorem~\ref{T5.1} that $\Phi_{\G,\SS}^\chi (a,d)$ is symmetric in 
$a,d$, and we can rewrite the sum over cusps $\aa$ in~\eqref{TFS2} as follows
  \be\label{1.7}\sum_{a,d>0}\min(a,d)^{k-1}\Phi_{\G,\SS}^\chi (a,d)\;.\ee
It is the function $\Phi_{\G,\SS}^\chi (a,d)$ that we will compute
for $\G=\G_0(N)$. 
\end{remark*}

\subsection{Explicit trace formulas in terms of class numbers}\label{sec1.3}
The class numbers appear 
naturally in terms of the conjugacy class invariants $\e(\X)$, and we first extend 
both the usual and the Kronecker-Hurwitz class numbers to all integers. 

Let 
$\M_n\subset M_2(\Z)$ denote the set of matrices of determinant $n$.  
The $\G_1$-equivariant bijection $M=\sm abcd \leftrightarrow Q_M(x,y)=cx^2+(d-a)xy-by^2 $
between matrices of determinant $n$ and trace $t$ and binary quadratic forms of discriminant 
$t^2-4n$ implies that for all $D$ and for $u\ge 1$:
\begin{equation}\label{13}
\sum_{\substack{\X\subset\ov{\M}_n\\ \DD(\X)=D, u|G_\X }} \e(\X)  = 
\begin{cases}
-2 H(-D/u^2) & \text{ if  } \tr(\X)\ne 0 \\
-H(-D/u^2) & \text{ if  } \tr \X= 0 
\end{cases}\;,
\end{equation}
where $G_\X$ is the content of the quadratic form $Q_{M}$ associated to any 
representative $M$ of $\X$, and $H(D)$ is the 
Kronecker-Hurwitz class number for $D\ge 0$, as extended by Zagier to 
all~$D$~\cite{Z1, Z}. We recall that if $D>0$, $H(D)$ is the the number of 
$\G_1$-equivalence classes of positive definite binary quadratic forms of 
discriminant $-D$, the forms with a stabilizer of order 2 or 3 in $\G_1$ being 
counted with multiplicity $1/2$ or $1/3$; $H(0)=-1/12$; and if $D<0$, $H(D)$ is 
$-u/2$ if $D = -u^2$ with $u \in \Z_{>0}$, and it is 0 if $-D$ is not a perfect 
square. 

Similarly, let $h(D)$ be the number of \emph{primitive} quadratic forms of discriminant $D<0$, and set
  \[h_0(D)=\frac{2h(D)}{w(D)},\] 
where $w(D)$ is the number of units of the quadratic order of discriminant~$D$. 
We extend the definition to all $D$ by setting
$h_0(0)=-1/12$, $h_0(u^2)=-\vp(u)/2$ if $u>0$ is an integer, and $h_0(D)=0$ if $D$ is not a negative discriminant or a square. 
This extension of $h_0(D)$ is compatible to that of $H(D)$, in the sense that for all $D$ we have
  \be\label{inv} H(-D)=\sum_{d^2|D} h_0\Big( \frac{D}{d^2}\Big), 
  \quad h_0(-D)= \sum_{d^2| D} H\Big( \frac{D}{d^2}\Big)\mu(d),  \ee
with the understanding that for $D=0$ only  
$d=1$ is included in both sums. It follows that for all $D\ne 0$ and for $u\ge 1$ we have 
\be\label{13.0}
\sum_{\substack{\X\subset\ov{\M}_n\\ \DD(\X)=D, G_\X=u }} \e(\X)  = 
\begin{cases}
-2 h_0(D/u^2) & \text{ if  } \tr(\X)\ne 0 \\
-h_0(D/u^2) & \text{ if  } \tr \X= 0 
\end{cases}\;, \ee
and with the convention that $0/0^2=0$ this formula also holds for $D=0$ and $u=0$. 

The class numbers enter the trace formula under the following assumption on the 
function~$\cc_{\G,\SS}^\chi$ in Theorems~\ref{T2} and~\ref{TSS}. Let 
$G_M=\gcd(c,d-a,b)$ be the content of the quadratic form $Q_M$ associated to the matrix $M$.
\begin{hypo}\label{h}
If $\G$ is a congruence subgroup of level $N$, then the 
function  $\cc_{\G, \SS}^\chi(M)$ in Theorem~\ref{T2} only depends on the conjugacy class invariants 
$\tr M$, $\det M$, and $(N,G_M)$, namely there is an arithmetic function $B_{\G, \SS}^\chi(u,t,n)$ such that
\[ \cc_{\G, \SS}^\chi(M)=B_{\G, \SS}^\chi((N,G_M), \tr M, \det M 
),\quad \text{ for all } M\in \M. \]
\end{hypo}
The assumption is natural, as the function~$\cc_{\G,\SS}^\chi$ is constant 
on~$\G_1$-conjugacy classes in $\M$, and we will see that it is satisfied 
for the double cosets giving the action of the usual Hecke and Atkin-Lehner 
operators for~$\G=\G_0(N)$.

Moebius inversion gives a function $C_{\G, \SS}^\chi(u,t,n)$ such that 
  \be\label{ass2}  B_{\G, \SS}^\chi(u,t,n)=\sum_{d|u} C_{\G, \SS}^\chi(d,t,n),
  \quad C_{\G, \SS}^\chi(u,t,n)=\sum_{d|u} B_{\G, \SS}^\chi(u/d,t,n)\mu(d),\ee
where we assume $u|N$ in both formulas. 
Note that the functions $B_{\G, \SS}^\chi, C_{\G, \SS}^\chi$ are only defined
on triples $(u,t,n)$ with $u|N$, $u^2|t^2-4n$, and they scale by $(-1)^k$ when $t$ is replaced by $-t$. 

Let $\SS\subset \M_n$ be a double coset, with $\M_n\subset \M$ consisting of 
matrices of determinant $n$. Under Assumption~\ref{h}, the sum over $\G_1$-conjugacy classes
  \be\label{1.17}\mathrm{RHS}:= \sum_{\X\subset\ov{\M}_n}\!\!\! 
  p_{k-2}(\tr M_\X, n) \; \cc_{\G, \SS}^\chi(M_\X)\; \e(\X)\ee
in the right hand side of \eqref{TF2} becomes\footnote{ In all formulas in this 
paper we adopt the convention that arithmetic functions are zero on nonintegers. 
For example the sums over $u$ in \eqref{1.1} and \eqref{1.2} are restricted to 
$u^2|t^2-4n$.}
  \be \begin{aligned}\label{1.1}
  \mathrm{RHS}
  &= \sum_{t\in \Z/\{\pm 1\}} p_{k-2} (t, n)\sum_{u|N}C_{\G, \SS}^\chi(u,t,n)
  \hspace{-.4cm}\sum_{\substack{\X\subset \ov{\M}_n \\ \tr(\X)=t, u| G_\X}} \hspace{-.4cm} \e(\X)\\
  &= -\sum_{t\in\Z}  p_{k-2} (t, n) \sum_{u|N}H\Big(\frac{4n-t^2}{u^2}\Big)C_{\G, \SS}^\chi(u,t,n).
  \end{aligned}\ee
In the second equality we used \eqref{13}, and the fact that $C_{\G, \SS}^\chi$ scales by $(-1)^k$ when 
$t$ is replaced by~$-t$, just like $p_{k-2} (t, n)$. 

Similarly using the extended class numbers $h_0(D)$ from~\eqref{13.0} we have 
  \be \begin{aligned}\label{1.2}
  \mathrm{RHS}
  &=\sum_{ t\in \Z/\{\pm 1\} } p_{k-2} (t, n)\sideset{}{^{\prime}}\sum_{u^2|t^2-4n}
  B_{\G, \SS}^\chi((N,u),t,n) 
  \hspace{-.4cm} \sum_{\substack{\X\subset \ov{\M}_n \\ \tr(\X)=t, G_\X=u}} \hspace{-.4cm}  \e(\X)  \\
  &= -\sum_{t\in \Z}  p_{k-2} (t, n) \sum_{u}h_0\Big(\frac{t^2-4n}{u^2}\Big)B_{\G, \SS}^\chi((N,u),t,n),
  \end{aligned}\ee 
with the understading that only $u=N$ is to be considered in the summation 
if $t^2-4n=0$. Note that if we restrict 
the sum over $\X$ in \eqref{1.17} to elliptic conjugacy classes, the range of 
summation in $t$ in the last terms of \eqref{1.1} and \eqref{1.2} becomes 
$t^2<4n$. Thus we obtain two explicit expressions in terms of class numbers of 
the right hand side of~\eqref{TFS2}.

\begin{remark}\label{r1.2}
The case $k=2$ of the trace formula reduces to the Kronecker-Hurwitz 
class number formula. Indeed, for $k=2$, $\G=\G_1$, $\wc=1$, 
and $\SS=\M_n$, the left side of~\eqref{TF2} vanishes, 
and using~\eqref{1.1} we obtain the Kronecker-Hurwitz relation: 
  $\sum_{t\in\Z} H(4n-t^2)=\sigma_1(n)\;.$ 
Taking~$\SS$ to be any $\G_1$-double coset instead of $\M_n$, the case $k=2$, 
$\G=\G_1$ of the trace formula~\eqref{TF2} gives a group-theoretical 
way of writing the Kronecker-Hurwitz formula in terms of the invariants $\e(\X)$:
  \[ \sum_{\X\subset \ov{\SS}} \e(\X) = -|\G_1\backslash\SS|\;,   \]
with the sum over $\ov{\G}_1$-conjugacy classes in $\ov{\SS}$. Together with 
D. Zagier, we give an elementary proof of a refinement of this formula
in~\cite{PZ1}.
\end{remark}

\subsection{Trace formulas on $\G_0(N)$ and $\G_1(N)$.}
We now specialize $\G=\G_0(N)$. Let~$k\ge 2$, $\chi$ 
a character modulo $N$ with $\chi(-1)=(-1)^k$, and let $\SS=\Delta_n$ be 
the double coset of the usual Hecke operator acting on 
$S_k(N,\chi):=S_k(\G,\chi)$, given in \eqref{delta}. 
Changing notation $\cc_{N, \chi}(M)=\cc_{\G,\SS}^\chi(M)$, we have
\be\label{6}
\cc_{N, \chi}(M)=\sum_{\substack{A\in \G\backslash\G_1 \\ AMA^{-1} \in \Delta_n}}
\chi(a_{AMA^{-1}})\;.
\ee
This function was computed by Oesterl\'e \cite{O77} (Lemma \ref{L9} below),
who showed that it only depends on the conjugacy class invariants
$t=\tr M$, $n=\det M$, and $u=(G_M,N)$: 
  \[ \cc_{N, \chi}(M)=B_{N,\chi}(u,t,n):=\frac{\varphi_1(N)}{\varphi_1(N/u)}\sum_{x\in S_{N}(u,t,n)} \chi(x), \]
where $S_{N}(u,t,n)=\{\alpha\in (\Z/N\Z)^\times: \ \alpha^2-t\alpha +n\equiv 0 \pmod{Nu}
\}$,\footnote{This set is well defined: if $\alpha\in \Z$ satisfies $\alpha^2-t\alpha +n\equiv 0 \pmod{Nu}$, 
so does $\alpha+Nd$ for every $d$, because of the assumption  $u|N$,
$u^2|t^2-4n$.} and $\varphi_1(N)$ is the index of
$\G_0(N)$ in $\G_1$, equal to $N\prod_{p|N}(1+1/p)$. Notice that $B_{N,\chi}(u,t,n)$ is multiplicative in $N$, 
so its Moebius inverse 
  \be\label{16} C_{N,\chi}(u,t,n):=\sum_{d|u} B_{N,\chi}(u/d,t,n)\mu(d)  \ee
is also multiplicative, and it can be easily computed numerically. 

We can now state the trace formula on the cuspidal subspace, which follows 
from Theorem~\ref{TSS} after expressing the right hand side as in~\eqref{1.1}, 
and using the formula for the cuspidal sum~$\Phi_{N,\chi}(a,d)$ proved
in~\S\ref{sec5.1}. Note that the  simpler formula for 
the linear combination~\eqref{0}, stated in the introduction, holds also 
for characters $\chi$ with $\chi(-1)\ne (-1)^k$, when both sides vanish. 
\begin{thm} \label{T3} 
Let $N\ge 1$ and $k\ge 2$ be integers, and $\chi$ a character mod $N$ with
 $\chi(-1)=(-1)^k$. With the function $C_{N,\chi}(u,t,n)$ defined above, we have
\begin{equation*}
\begin{split} 
\tr(T_n,S_k(N,\chi))= 
\;-\; \frac{1}{2}\sum_{t^2\le 4n}p_{k-2}(t,n)\cdot
\sum_{u|N} H\Big(\frac{4n-t^2}{u^2}\Big) C_{N,\chi}(u,t,n)\\
-\frac 12 \sum_{n=ad}\min(a,d)^{k-1}\Phi_{N,\chi}(a,d)
+\delta_{k,2}\delta_{\chi,{\bf 1}}\sigma_{1,N}(n)\;,
\end{split}
\end{equation*}
where $\; \sigma_{1,N}(n)=\sum_{d|n,(N,d)=1} n/d\; $, and 
  $$\Phi_{N,\chi}(a,d)=\sum_{\substack{N=rs\\(r,s)|(N/c(\chi),a-d )}} 
  \varphi((r,s))\; \chi(\alpha_{r,s}^{a,d}),$$ 
where $\alpha=\alpha_{r,s}^{a,d}$ is the residue class modulo $N/(r,s)$ such that
$\alpha\equiv a \pmod{r}, \alpha\equiv d \pmod{s}$, $c(\chi)$ is the conductor of
$\chi$, and $\varphi$ denotes Euler's function.
\end{thm}
The terms for $t^2=4n$ are explictly computed in Remark~\ref{r4.6}: 
they are nonzero only when $n$ is a square, when they contribute 
$\frac{\vp_1(N)}{12}(k-1)n^{k/2-1}\chi(\sqrt{n})$. 

\begin{remark*}An equivalent formula for the cuspidal trace was obtained by 
Cohen and Oesterl\'e by analytic means~\cite{C77, O77}, with the sum over $u$ written as 
in~\eqref{1.2} in terms of $B_{N,\chi}$.
\end{remark*}

We also give a formula for the trace of $T_n\circ W_\ell$ on $S_k(N)$,
with $W_\ell$ the Atkin-Lehner operator for an exact divisor $\ell$ of $N$.  
Let $C_N(u,t,n)=C_{N, \bf{1}}(u,t,n)$ in~\eqref{16}, with ${\bf 1}$ the 
trivial character mod~$N$. A different formula for 
$\tr(T_n\circ W_\ell,S_k(N))$ was given by Skoruppa and Zagier~\cite{SZ}, 
but assuming that $(n,N)=1$. 
\begin{thm} \label{T4}
Let $N=\ell \ell' $ with $(\ell,\ell')=1$ and $k\ge 2$ even, $w=k-2$. For \emph{all} $n\ge 1$ we have
\[\begin{split}
\tr(T_n\circ W_\ell,S_k(N))=
-\frac{1}{2}\sum_{\substack{t^2\le 4 \ell n \\ \ell|t}}\!\! \frac{p_{w}(t,\ell n)}{\ell^{w/2}} 
\cdot  \sum_{\substack{u|\ell\\ u'|\ell'}} H\Big(\frac{4\ell n-t^2}{(uu')^2}\Big) C_{\ell'}\big(u',t,\ell n\big) \mu(u)\\
-\frac 12  \sum_{\substack{n\ell=ad\\ \ell|a+d}} \frac{\min(a,d)^{k-1}}{\ell^{w/2}}\Phi_{N,\ell}(a,d)
+\delta_{k,2}\sigma_{1,N}(n)\;,
\end{split}\]
where
\[\Phi_{N,\ell}(a,d)=\frac{\varphi(\ell)}{\ell}\sum_{\substack{\ell'=rs,\ (r,s)|a-d\\   
(r,a)=1, (s,d)=1}} \varphi((r,s))\;.\] 
\end{thm}
The terms for $t^2=4\ell n$ in the summation above are present only if 
$\ell=1$, $n$ is a square, and $(n,N)=1$, when they contribute 
$\frac{\vp_1(N)}{12}(k-1)n^{w/2}$. Note that $\Phi_{N,1}$ is the 
same as the function $\Phi_{N,\chi}$ in Theorem \ref{T3} for $\chi={\bf 1}$.

The function $C_{N}(u,t,n)$ is explicitly computed in Lemma~\ref{L4}, where we show that
  \be \label{1.23}  C_{N}(u,t,n)= |S_N(t,n)|\cdot C_N(u,t^2-4n),\ee
with $S_N(t,n)=\{\alpha\in (\Z/N\Z)^\times: \ \alpha^2-t\alpha +n\equiv 0 \pmod{N}\}$, 
and $C_N(u,D)$ an explicit multiplicative function in $(N,u)$. For example, when
$N$ is square-free, we have $C_N(u,D)=u$ independent of $D$.

We also obtain a trace formula for $\Gamma_1(N)$, by summing over $\chi$ the 
formula in Theorem~\ref{T3}. For $u|N$, $u^2|t^2-4n$ we let 
\[B_N(u,t,n)=\frac{1}{\vp(N)}\sum_\chi B_{N,\chi}(u,t,n)=
\begin{cases} \dfrac{\vp_1(N)}{\vp_1(N/u)} & \text{if } Nu|t-n-1 \\
                \phantom{xxx}0 & \text{otherwise}
             \end{cases}
\]
where the sum is over all the $\vp(N)$ characters modulo $N$. Let $D_N(u,t,n)$ denote its Moebius inverse defined as in~\eqref{16}.  
\begin{thm} \label{T5} Let $N\ge 1$, $k\ge 2$, $n \ge 1$, and set $\Gamma=\G_1(N)$.  
We have:
\[
\tr(T_n,S_k(\Gamma)+M_k(\Gamma))=\delta_{k,2}\sigma_{1,N}(n)-
\varphi(N)\cdot\sum_{\substack{t\in\Z \\ N|t-n-1  } }p_{k-2}(t,n)\cdot
\sum_{u|N} H\left(\tfrac{4n-t^2}{u^2}\right) D_N(u,t,n)\;,
\]
\comment{
\[\tr(T_n, S_k(\Gamma)+M_k(\Gamma))=\delta_{k,2}\sigma_{1,N}(n)-
\varphi(N)\cdot\sum_{
t\in\Z }p_{k-2}(t,n)\cdot
 \sum_{\substack{ u\\N u_N|(t-n-1)}} h_0\left(\tfrac{t^2-4n}{u^2}\right) 
 \psi_N(u_N),
\] 
where $u_N:=\gcd(N,u)$ and for $u|N$ we set $\psi_N(u)=\frac{\varphi_1(N)}{\varphi_1(N/u)}$.
The summation over $u$ contains only the term $u=N$ if $t^2-4n=0$. 
}
\[\begin{split}
\tr (T_n, S_k(\G))= \delta_{k,2}\sigma_{1,N}(n)-\frac {\varphi(N)}2 
\cdot\sum_{t^2\le 4n  }p_{k-2}(t,n)\cdot
\sum_{u|N} H\left(\tfrac{4n-t^2}{u^2}\right) D_N(u,t,n)\\
-\frac 14\sum_{n=ad}\min(a,d)^{k-1}\left(\Psi_N(a,d)+(-1)^k\Psi_N(-a,-d)\right)
\;,
 \end{split}
\]
where 
\[\Psi_N(a,d)=\sum_{\substack{N=rs \\ r|(a- 1), s|(d-1) } } 
\varphi((r,s))\varphi(N/(r,s)). \]
\end{thm}

The formulas in Theorems \ref{T3}, \ref{T4} and \ref{T5} were
verified numerically for a large range of parameters.

\subsection{Applications}\label{sec1.4}

An important application of the trace formula is that it can be used to
construct explicitly modular forms. Indeed the ``trace form''
\[\sum_{n \ge 1 } \tr(T_n, S_k(N,\chi)) e^{2\pi inz}, \]
is a nonzero modular form belonging to $S_k(N,\chi)$, and 
by acting on it with Hecke operators one can generate the entire space. 

We give two explicit examples for $N=4$. In this case it is easier to apply 
Theorem~\ref{T5}, since $\tr(T_n, S_k(\G_1(4))$ equals the trace on $S_k(4)$
for $k$ even, and the trace on $S_k(4,\chi_4)$ for $k$ odd, with~$\chi_4$ the 
nontrivial character modulo 4. The functions $D_N(u,t,n)$ in Theorem~\ref{T5} 
are straightforward to compute, and we obtain the following explicit formulas. 
More details are given in Section~\ref{sec4.3}.
\begin{cor} \label{C2} \emph{a)} Let $n\ge 1$ be odd. If $k\ge 2$ is even we have
    \[ \tr(T_n,S_k(4))=-3 \sum_{s^2\le n} p_{k-2}(2s,n) H(n-s^2) 
    -\frac 32 \sum_{n=ad} \min(a,d)^{k-1}+\delta_{k,2}\sigma_1(n)\;, \]
  while if $k\ge 3$ is odd we have
  \[\begin{split}
\tr(T_n,S_k(4,\chi_4))=-
  \sum_{\substack{s^2\le n\\ s\ \mathrm{odd}}} (-1)^{(2s-n-1)/4}p_{k-2}(2s,n) 
  \left(H(n-s^2)+2H\left(\frac{n-s^2}4\right)\right)\\
  -\sum_{\substack{n=ad\\4|a-d}} \min(a,d)^{k-1}\chi_4(a)\,.
\end{split}\]

\emph{b)} Let $n\ge 2$ be even. For $k\ge 2$, and $\chi$ the character mod 4 with 
$\chi(-1)=(-1)^k$ we have
\[\tr(T_n,S_k(4,\chi))= -\sum_{\substack{t^2\le 4n\\ 4|t-n-1}}
p_{k-2}(t,n) H(4n-t^2)-\sum_{\substack{n=ad\\ a\ \mathrm{odd}}}\chi(a)
\min(a,d)^{k-1}
+\delta_{k,2}\sum_{\substack{n=ad\\ a\ \mathrm{odd}}} d\,.
\]
\end{cor}
\begin{remark*}
Note that in both sums over $s$ in part a) we have $4|2s-n-1$ in order for 
$H(n-s^2)\ne 0$. Note that $\tr(T_n, S_k(4,\chi_4))=0$ for 
$n\equiv 3 \pmod 4$, as all terms in the formula vanish.
\end{remark*}
For $f=\sum_{n\ge 1} a_n q^n$ a modular form in  $S_k(4)$, respectively 
in $S_k(4,\chi_4)$  it is easy to see that both $\sum_{n\ \mathrm{odd}} a_n q^n$
and $\sum_{n\ \mathrm{even}} a_n q^n$
belong to $S_k(4)$,  respectively to $S_k(8,\chi_4)$. Applying this 
observation to the trace form, we obtain explicit cusp forms of the type 
given in the introduction. For $k$ even and odd index traces,  
this proves a conjecture of Cohen~\cite{Co}, which was recently proved 
by Mertens by different methods~\cite{Me1}. For $k$ odd, and for $k$ even 
and even index traces, the resulting formulas seem to be new.

Since the spaces $S_2(4)$, $S_3(4, \chi_4)$ are trivial, the formulas in the 
corollary reduce to class number relations similar to the Kronecker-Hurwitz 
formula. Other such relations can be obtained from Theorem~\ref{T4} taking 
small values for~$N$.

As a second application, we use Theorem~\ref{T5} to obtain the limit stated in 
the introduction for the trace of a fixed Hecke operator $T_n$ on $\G_1(N)$ when 
$N$ goes to infinity. In fact we obtain precise formulas for 
$N>2n+2$, which no longer contain class numbers: 
\begin{cor} \label{C1} Fix $k\ge 2$ and $n>1$. For $N>2n+2$ we have:
\[
\tr(T_n, S_k(\G_1(N))+M_k(\G_1(N)))=\delta_{k,2}\sigma_{1,N}(n)+
\frac{\varphi(N)}{2}\frac{n^{k-1}-1}{n-1}\sum_{u|n-1} \varphi(u)
\frac {\vp_1(N)}{\vp_1(N/(u,N))}
\]
\[\tr(T_n, S_k(\G_1(N)))=\delta_{k,2}\sigma_{1,N}(n)-\frac 12 
\sum_{u|(N,n-1)}\vp\big((u,N/u)\big)\vp\left( \frac{N}{(u,N/u)}\right).
\]
 \end{cor}
Assuming $(N,n-1)=1$, the first formula gives:
\[\lim_{\substack{N\rar \infty\\ (N,n-1)=1}} 
\frac{\tr(T_n, S_k(\G_1(N))+M_k(\G_1(N)))}{\varphi(N)}=\frac{n^{k-1}-1}{2},
\]
while the second gives the limit formula from in the introduction.
\begin{proof}Since $N|t-n-1$ in the sums in Theorem~\ref{T5}, the  
assumption $N>2n+2$ implies that only $t=n+1$ contributes in the first formula and 
the sum over $t$ in the second formula is empty. We also have $\Phi_N(-a,-d)=0$ for $n=ad$, $a,d>0$,
and since $p_{k-2}(n+1,n)=(n^{k-1}-1)/(n-1)$ the conclusion follows. 
\end{proof}

\section{General trace formulas on the Eisenstein subspace and on the 
cuspidal subspace}\label{secEis}

In Section \ref{s4.1} we take $\G$ to be Fuchsian subgroup of the 
first kind with cusps and we compute the trace of a double coset operator~$[\SS]$ 
on the Eisenstein subspace~$E_k(\G,\chi)$. Here $\SS$ is any double coset 
contained in the commensurator of $\G$ inside $\GL_2^+(\R)$. As an immediate 
consequence, in Section~\ref{s4.2} we use an equivalent formulation
of Theorem~\ref{T2} that makes sense for a Fuchsian group (see ~\eqref{TFF})
to obtain the trace formula on the cuspidal subspace in Theorem~\ref{TSS}.  

The trace formulas on the Eisenstein and on the cuspidal subspaces depend on 
the arithmetic function $\Phi_{\G,\SS}^\chi(a,d)$, and in Section~\ref{sec4} 
we give a practical way to compute this function when~$\G$ is a finite index 
subgroup of~$\G_1$ and $\SS$ is a double coset satisfying the assumption in 
Theorem~\ref{T2}, which we restate here as follows. 
\begin{hypo}\label{eq_star} The map 
$$\G\backslash\SS\longrightarrow \G_1\backslash\G_1\SS,\quad   \G \sigma \mapsto
 \G_1\sigma  $$  
is bijective,  or equivalently $|\G\backslash\SS|= |\G_1\backslash\G_1\SS|$.
\end{hypo}

\subsection{A trace formula on the Eisenstein subspace}\label{s4.1}
We start by introducing some notation and terminology related to the cusps. For 
a parabolic or hyperbolic matrix $\ss\in \GL_2(\R)^+$, we denote by 
$\sgn(\ss)\in \{\pm 1\}$ the sign of the eigenvalues of $\ss$. For $\aa$ a cusp 
of $\G$, we let $\G_\aa\subset\G$ be the stabilizer of $\aa$ in $\G$. Thus 
$\G_\aa=\pm \la\g_\aa \ra$ if $-1\in \G$, and $\G_\aa=\la \g_\aa\ra $ if 
$-1\notin \G$, for a generator $\g_\aa\in \G_\aa$, with $\sgn(\g_\aa)=+1$ if $-1 
\in \G$. Let $C_\aa\in \SL_2(\R)$ be a scaling matrix for the cusp $\aa$, namely 
$C_\aa\aa =\infty$, and \be \label{5.0} C_\aa \g_\aa C_\aa^{-1}= \sgn(\g_\aa) 
\sm 1101 \,. 
\ee 
We assume that scaling matrices for equivalent cusps satisfy $C_{\g\aa}=C_\aa\g^{-1}$. 

For $\ss\in\widetilde{\G}$ (the commensurator of $\G$), there exists $n\in\Z$ such that $\g_\aa^n\in \ss^{-1}\G\ss $, and since 
$\ss \g_\aa^n \ss^{-1}\in \G$ is a parabolic element fixing $\ss\aa$, we have that $\bb=\ss\aa$ is also a cusp of $\G$, 
and 
  \be\label{5.2} \ss \g_\aa^n \ss^{-1}=\pm \g_\bb^m \;, \ee
for some $m\in \Z$. 
As $C_\bb \ss C_\aa^{-1}\infty=\infty$ we have 
  \be \label{5.31} C_\bb \ss C_\aa^{-1}=\pmat {a_\aa(\ss)}b0{d_\aa(\ss)}\;,\ee
for some $a_\aa(\ss),d_\aa(\ss)\in \R$, and therefore 
 $C_\bb\ss \g_\aa\ss^{-1} C_{\bb}^{-1}=\pm \sm 1{a_\aa(\ss)/d_\aa(\ss)}01.$
Raising the previous relation to the power $2n$ and using \eqref{5.2} we obtain that the ratio 
 \be \label{5.4} \frac {a_\aa(\ss)}{d_\aa(\ss)} =\frac mn  \ee 
is a positive rational number. 
\begin{remark}\label{r2}
For fixed $\aa$, the constants $a_\aa(\ss)$, $d_\aa(\ss)$ only depend 
on the coset $\G\ss$ since $C_{\g\bb}=C_\bb\g^{-1}$. Moreover $a_\aa(\ss)$, $d_\aa(\ss)$ are invariant under
the map $(\aa, \ss)\mapsto (\g \aa, \ss\g^{-1})$, for~$\g\in\G$. It follows that by scaling the double coset 
$\SS$ we can assume that $a_\aa(\ss), d_\aa(\ss)\in \Z$ for all cusps $\aa$ and $\ss\in\SS$.
\end{remark}

\subsubsection{Constant terms of Eisenstein series}\label{s4.1.1}
For an Eisenstein series $E\in E_k(\G,\chi)$, the constant term $A_E(\aa)$ of $E$ at the cusp $\aa$ is defined as the constant term 
of the Fourier expansion of $E|_k C_\aa^{-1}$:
 \[ A_E(\aa)=a_0(E|_k C_\aa^{-1})=\lim_{z\rightarrow i\infty} E|_k C_\aa^{-1} (z) \;. \]
From \eqref{5.0} we have $A_E(\aa)=a_0(E|_k C_\aa^{-1}T^{-1})=\chi(\g_\aa)^{-1}\sgn(\g_\aa)^k a_0(E|_{k} C_\aa^{-1}) $,
so $A_E(\aa)$ vanishes unless $\chi(\g_\aa)=\sgn(\g_\aa)^k$. Therefore we define the  $\G$-invariant set
  \[ \cps(\G,\chi)=\{\aa\in \cps(\G)\;:\; \chi(\g)=\sgn(\g)^k \text{ for } \g\in\G_\aa    \}, \]
and we let $C(\G,\chi)\subset C(\G) $ be sets of representatives 
for $\G$-equivalence classes in $\cps(\G,\chi)$, respectively in $\cps(\G)$. 
When $\chi={\bf 1}$, the trivial character, we have $C(\G,{\bf 1})= C(\G)$ if 
$k$ is even, while $C(\G,{\bf 1})$ is the set of regular cusps if $k$ is odd and $-1\notin\G$. 

Since  $C_\aa \g^{-1}$ is a scaling matrix for $\g\aa $ for $\g\in\G$, it follows that 
$A_E({\g\aa})=\chi(\g)A_E({\aa})$. Identifying the vector space $\C^{|C(\G,\chi)|}$ with the space of maps
$f:C(\G,\chi)\rightarrow \C$, we have an injective map  
  \be\label{5.9} E_k(\G,\chi)\longrightarrow \C^{|C(\G,\chi)|},\quad E\mapsto A_E\;. \ee
This map is a bijection, unless $k=2$ and $\chi=\bf 1$, when $C(\G,{\bf 1})=C(\G)$ and we have an exact sequence
  \be\label{5.10}  0\xrightarrow{\quad \; \quad }E_2(\G)\xrightarrow{\;\;\; E \mapsto A_E\;\;\;} \C^{|C(\G)|} 
  \xrightarrow{f\mapsto\sum_\aa f(\aa)} \C \xrightarrow{\quad \; \quad } 0\; .\ee

We can now compute the constant terms of $E|[\SS]$, for $\SS\subset \widetilde{\G}$ a double coset. 
For a cusp $\aa\in \cps(\G,\chi)$ and $E\in E_k(\G,\chi)$ we have by \eqref{hecke}
  \[ 
  E|[\SS]|_k C_\aa^{-1}= \sum_{\ss \in \G\backslash\SS} \det\ss^{k-1} \wc(\ss) E|_k\ss C_a^{-1}\;.
  \]
Replacing $\ss C_\aa^{-1}$ from \eqref{5.31} in the previous relation, and taking $z\rightarrow i \infty$ we obtain
  \be \label{5.5} a_0(E|[\SS]|_k C_\aa^{-1})=\sum_{\ss \in\G\backslash \SS} 
  a_0(E|_{k}C_{\ss\aa}^{-1} )\; \frac{a_\aa(\ss)^{k-1}}{d_\aa(\ss)}\; \wc(\ss) \;.\ee

\subsubsection{The trace formula}
Using~\eqref{5.5}, we compute $\tr([\SS], E_k(\G,\chi))$ in the next theorem. First we
prove a lemma interesting in its own right, which is needed for the case $k=2$, $\chi={\bf 1}$, and whose
proof will be used in the proof of the theorem. We make the following assumption 
on the double coset $\SS\subset \wc$, which is implied but much weaker than assumption~\ref{eq_star}.
\begin{assumpt}\label{ass4}
   If $-1 \notin \G$, then $\SS\cap -\SS=\emptyset$. 
\end{assumpt}

\begin{lemma}\label{L5.1}
Let $\G$ be a Fuchsian group of the first kind, and let $\SS\subset \widetilde{\G}$ 
be a double coset satisfying Assumption~\ref{ass4}. Then 
  \[ \sum_{\ss\in \G\backslash \SS} \frac{a_\aa(\ss)}{d_\aa(\ss)}= |\G\backslash\SS|, \]
independent of the cusp $\aa$ of $\G$, where $C_{\ss\aa} \ss C_\aa^{-1}=\pmat {a_\aa(\ss)}*0{d_\aa(\ss)}$.
\end{lemma}
\begin{proof}
For each $\bb\in C(\G)$, let $\SS_{\aa\bb}:=\{\ss\in \SS \;:\; \ss\aa=\bb \}$. Each coset $\G\ss$ contains 
a representative $\ss_0$ with $\ss_0\aa=\bb$, where $\bb$ is the fixed representative in $C(\G)$ of the 
equivalence class of cusps $\G\ss\aa$, and if $\g\ss_0$ is another such representative we have $\g\in\G_\bb$. A similar
reasoning applies to right cosets, so we have the disjoint decompositions (with a slight abuse of notation) 
  \be\label{5.14}\G\backslash\SS=\bigcup_{\bb\in C(\G)} \G_\bb \backslash \SS_{\aa\bb}, \quad
  \SS/\G=\bigcup_{\bb\in C(\G)} \SS_{\aa\bb}/\G_\aa.
  \ee
For $a,d>0$, let 
  \be\label{5.20}
  \SS_{\aa\bb}(a,d)=\left\{\ss\in\SS_{\aa \bb} \;:\; C_\bb \ss C_\aa^{-1} =\sgn(\ss) \pmat a*0d \right\}
  = \SS_{\aa\bb}^+(a,d)\cup\SS_{\aa\bb}^-(a,d) \;,\ee
where $\SS_{\aa\bb}^\pm(a,d)$ consist of those $\ss\in \SS_{\aa\bb}(a,d)$ having $\sgn(\ss)=\pm 1$. 
The first decomposition gives
  \be \label{5.11}\sum_{\ss\in \G\backslash \SS} \frac{a_\aa(\ss)}{d_\aa(\ss)}=
  \sum_{a,d>0} a \sum_{\bb\in C(\G)}\frac 1d \cdot |\G_\bb\backslash\SS_{\aa\bb}(a,d)| \;.
  \ee
The set $\SS_{\aa\bb}(a,d)$ is left invariant by~$\G_\bb$ and right invariant by~$\G_\aa$, 
and we show that 
 \be\label{5.15} \frac 1d\cdot  | \G_\bb\backslash\SS_{\aa\bb}(a,d)|=
 \frac 1{(a,d)} \cdot | \G_\bb\backslash\SS_{\aa\bb}(a,d)/\G_\aa|=
 \frac 1a \cdot |\SS_{\aa\bb}(a,d)/\G_\aa|\;.  \ee
To prove this identity, we assume by Remark~\ref{r2} that $a,d\in\Z$. We also 
assume  for simplicity that $-1 \in \G$, the other case being similar 
(using Assumption~\ref{ass4}). We have 
$\G_\bb\backslash\SS_{\aa\bb}(a,d)= \la \g_\bb\ra \backslash \SS_{\aa\bb}^+(a,d)$, and multiplying 
$\ss_b=C_{\bb}^{-1}\sm ab0d C_\aa\in \SS_{\aa\bb}^+(a,d)$ on the left by $\g_\bb^n$ and on the right by $\g_\aa^m$ 
changes $b\mapsto b+ma+nd$. Therefore a set of representatives for 
$\la \g_\bb\ra \backslash \SS_{\aa\bb}^+(a,d)/\la \g_\aa\ra$ is 
  \[\{\ss_b\in \SS_{\aa\bb}^+(a,d) \;:\; 0\le b<(a,d)\}\;, \]
while a set of representatives for $\la \g_\bb\ra \backslash \SS_{\aa\bb}^+(a,d)$, respectively 
$ \SS_{\aa\bb}^+(a,d)/\la \g_\aa\ra$, is the same set, with the range for $b$ 
replaced by $0\le b<d$, respectively $0\le b<a$, proving~\eqref{5.15}. 

Using~\eqref{5.15}, formula \eqref{5.11} becomes 
 \[\sum_{\ss\in \G\backslash \SS} \frac{a_\aa(\ss)}{d_\aa(\ss)}=
 \sum_{a,d>0} a \sum_{\bb\in C(\G)} \frac 1a \cdot |\SS_{\aa\bb}(a,d)/\G_\aa|=|\SS/\G|\;,
\]
by the second decomposition in \eqref{5.14}, and  the claim follows from the equality $|\SS/\G|=|\G\backslash \SS|$. 
\end{proof}
We now introduce the cuspidal sum entering the trace formula on the Eisenstein 
subspace. Denote by $\SS_\aa(a,d)$ the set $\SS_{\aa\aa}(a,d)$ introduced in~\eqref{5.11}
and let~$\SS_\aa=\SS_{\aa\aa}$ be the stabilizer of the cusp $\aa$ in $\SS$. 
We define the arithmetic function $\Phi_{\G,\SS}^\chi(a,d)$ as in~\eqref{1.7} the introduction.
   
\begin{theorem} \label{T5.1} Let $\G$ be a Fuchsian group of the first kind, 
$k\ge 2$, and $\chi$ a character of $\G$ with $\chi(-1)=(-1)^k$ if $-1\in\G$. 
Let $\SS\subset \wg$ be a double coset satisfying Assumption~\ref{ass4}. 

\emph{(a)} With the function $\Phi_{\G,\SS}^\chi(a,d)$ defined above, we have
 \be\label{5.8} 
   \tr([\SS], E_k(\G,\chi)) 
         = \sum_{a,d>0} a^{k-1}\Phi_{\G,\SS}^\chi(a,d)-
         \dd_{k,2}\dd_{\chi,{\bf 1}} \sum_{\ss\in\G\backslash\SS} \wc(\ss) \;. 
  \ee
  
\emph{ (b)} The function $\Phi_{\G,\SS}^\chi(a,d)$ is symmetric in $a,d$ and for $a\ne d$ we have
 \be\label{5.7} \Phi_{\G,\SS}^\chi(a,d)=\frac{1}{|d-a|} \sum_{\ss\in H_{\G,\SS}(a,d)}\sgn(\ss)^k\wc(\ss)\ee
where $H_{\G,\SS}(a,d)\subset \SS$ is a system of representatives for the hyperbolic $\G$-conjugacy classes 
$\X\subset \ov{\SS}$ whose elements fix two cusps of $\G$, and that have eigenvalues $a,d$ or $-a, -d$. 
  \end{theorem}
\begin{proof} (a) By \eqref{5.5}, the action of $[\SS]$ on Eisenstein series corresponds to an action 
on $P\in \C^{|C(\G,\chi)|}$ given by 
  \be \label{5.13}P|[\SS] (\aa)=\sum_{\ss \in\G\backslash \SS} 
  P (\ss\aa) \; \frac{a_\aa(\ss)^{k-1}}{d_\aa(\ss)}\; \wc(\ss), \quad \aa\in \C^{|C(\G,\chi)|} \;,\ee
and we conclude
 \be \label{5.16}\tr([\SS], \C^{|C(\G,\chi)|})=\sum_{\aa\in C(\G,\chi)}
       \sum_{\ss\in\G_\aa\backslash \SS_\aa} \frac{a_\aa(\ss)^{k-1}}{d_\aa(\ss)}\wc(\ss)\;. \ee
If $k=2$ and $\chi={\bf 1}$, let $P_0\in \C^{|C(\G)|}$ such that 
$P_0(\aa)=1$ for all $\aa\in C(\G)$. We have that 
 \[ P_0|[\SS]= P_0 \cdot \!\! \sum_{\ss\in\G\backslash\SS}\!\! \wc(\ss)\; ;\]  
 indeed, we can assume without loss of generality that $\SS=\G\ss_0\G$ is a primitive double coset, so $\wc$ is constant
 on $\SS$ and by \eqref{5.13} we obtain that $P_0|[\SS](\aa)$ is given by the left hand side of the identity in Lemma \ref{L5.1}, multiplied
by $\wc(\ss_0)$. The exact sequence \eqref{5.10} then gives  
 \[\tr([\SS], E_k(\G,\chi))=\tr([\SS], \C^{|C(\G,\chi)|}) -\dd_{k,2}\dd_{\chi,{\bf 1}} \sum_{\ss\in\G\backslash\SS}\!\! \wc(\ss) \;. \]

We rewrite \eqref{5.16} as $\tr([\SS], \C^{|C(\G,\chi)|})=\sum_{a,d>0} a^{k-1}\Psi_{\G,\SS}^\chi(a,d)$, with 
\be \label{5.12}
\Psi_{\G,\SS}^\chi(a,d)=\frac{1}{d} \sum_{\aa\in C(\G,\chi)}
  \sum_{ \ss\in\G_\aa\backslash \SS_\aa(a,d)}\sgn(\ss)^k \wc(\ss)\;.
\ee
Since $\sgn(\ss)^k \wc(\ss)$ is invariant under $\ss\mapsto \g\ss$ and $\ss\mapsto \ss\g$, 
for $\g\in\G_\aa$ and $\aa\in C(\G,\chi)$, and $\SS_\aa(a,d)$ is also left and right invariant 
under multiplication by $\G_\aa$, by the first equality in~\eqref{5.15} it follows that 
$\Psi_{\G,\SS}^\chi=\Phi_{\G,\SS}^\chi$, proving~\eqref{5.8}.  

(b) That $\Phi_{\G,\SS}^\chi$ is symmetric follows from~\eqref{5.7}, 
since the right hand side is obviously symmetric in $a,d$. To prove~\eqref{5.7},
denote by $\Theta_{\G,\SS}^\chi(a,d)$ its right hand side. 
Each $\ss\in H_{\G,\SS}(a,d)$ fixes two cusps of $\G$, and exactly one of them, 
denoted $\aa$, satisfies $C_\aa \ss C_\aa^{-1}=\sgn(\ss)\sm ab0d$, 
independent of the scaling matrix $C_\aa$ used (for the other cusp, 
$a,d$ are reversed--see~\cite[p. 266]{Sh1}). By replacing $\ss$ by a conjugate,
we can assume that $\aa$ belongs to the set 
of cusp representatives $C(\G)$ fixed in~\S\ref{s4.1.1}, and  
$\pm\ss\in\SS_\aa (a,d)$, with the minus sign possible if and only 
if $-1\in \G$ by Assumption~\ref{ass4}. Denoting by 
$\SS_\aa'(a,d)$ either $\SS_\aa^+ (a,d)$ (defined in \eqref{5.20}) if $-1\in \G$, 
or $\SS_\aa (a,d)$ if $-1\not\in\G$, we have by Lemma~\ref{L4.4} that 
$\SS_\aa'(a,d)$ consists of hyperbolic elements fixing two cusps of $\G$, 
and we obtain
\[
\Theta_{\G,\SS}^\chi(a,d)=\frac{1}{|d-a|}\sum_{\aa\in C(\G)}\sum_{\ss\in \G_\aa \sslash \SS_\aa'(a,d)}\sgn(\ss)^k\wc(\ss)\;,
\]
where $\sslash$ denotes the conjugation action of $\G_\aa$ on  $\SS_\aa'(a,d)$, 
and the sum is over any system of representatives for the orbits of this action. 
The set $\SS_\aa'(a,d)$ is invariant under left and right multiplication
by the group $\la\g_\aa\ra $ generated by $\g_\aa$. Changing variables 
$\ss\mapsto \g \ss$ for $\g\in \la\g_\aa\ra$ in the sum over $\ss$, 
scales the sum by $\sgn(\g)^k\chi(\g)$. Therefore the inner sum 
vanishes, unless $\chi(\g)=\sgn(\g)^k$ for $\g\in \G_\aa$, that is 
unless $\aa\in C(\G,\chi)$. To show that $\Theta_{\G,\SS}^\chi=\Phi_{\G,\SS}^\chi$, 
it remains to 
prove that
  \[\frac{1}{|d-a|}\cdot|\G_\aa \sslash \SS_\aa'(a,d)|=\frac{1}{(a,d)}\cdot|\G_\aa\backslash \SS_\aa(a,d)/\G_\aa|\;, \] 
which follows by a similar argument as~\eqref{5.15}. 
\end{proof}

\subsection{A trace formula on the cuspidal subspace}\label{s4.2}
In order to extract from~\eqref{TF2} the Eisenstein contribution,
we will use an equivalent version, proved in~\cite{P}, which can be stated for an arbitrary Fuchsian
group of the first kind $\G$. 

For a $\ov{\G}$-conjugacy class $\X\subset  \GL_2^+(\R)/\{\pm 1\}$, we define 
the following analogue of the conjugacy class invariant $\e$ for $\G$: 
\[\e_\G(\X)=\begin{cases} \phantom{xx}
\dfrac{|\G\backslash\H| }{2 \pi} &  \text{if $M_X$ scalar,}\vspace{2mm} \\
\dfrac{\sgn \DD(\X)}{|\stab_{\ov{\G}} M_\X|} &  \text{ otherwise,}           
            \end{cases}
\]
where $|\G\backslash\H|$ is the area of a fundamental domain for $\G$ with respect 
to the standard hyperbolic metric, and we use 
the convention that $1/\infty=0$. Any double coset $\SS\subset \wg$ 
contains only finitely many conjugacy classes~$\X$ with $\e_\G(\X)\ne 0$, 
namely the elliptic, scalar, and those hyperbolic classes that 
contain an element fixing two distinct cusps of $\G$, for which $\e_\G(X)=1$
(see Lemma~\ref{L4.4}). We show in~\cite[Sec. 4]{P} that the 
trace formula in Theorem~\ref{T2} is equivalent to the following statement:
 \be \label{TFF} \begin{split}
 \tr([\SS], M_k(\G,\chi)+ S_k^c(\G,\chi)) \,=\,
 \sum_{\X\subset \ov{\SS}}p_{k-2}(\tr M_\X, \det M_\X)\;\wc(M_\X)\;\e_\G(\X)   \\
 \+ \delta_{k,2}\delta_{\chi,{\bf 1}}\; \sum_{\ss\in\G\backslash\SS} \wc(\ss)\;,
 \end{split}
 \ee
where the sum is over $\ov{\G}$-conjugacy classes $\X$ in $\ov{\SS}$ with 
representative $ M_\X\in \SS$. 

We now use Theorem~\ref{T5.1} to obtain a trace formula on the 
cuspidal subspace from~\eqref{TFF}, after first recalling a result of 
J. Oesterl\'e. 
\begin{lemma}[Oesterl\'e]\label{L4.4} Let $\G$ be a Fuchsian subgroup of the 
first kind and let $M\in \wg$ such that $\stab_{\ov{\G}} M$ is finite. 
Then~$M$ is either elliptic, or it is hyperbolic fixing two distinct cusps 
of~$\G$. In the latter case we have $|\stab_{\ov{\G}} M|=1$
\end{lemma}
\begin{proof}
More precisely, it is shown in~\cite[Proof of Theorem 2]{O77} that 
that any non-scalar $M\in \wg$ with $\tr^2(M)\ge 4\det(M)$ falls 
in one of three  cases: $M$ is parabolic fixing a cusp of~$\G$; 
$M$ is hyperbolic with the same fixed points as those of a hyperbolic matrix in $\G$; 
or $M$ is hyperbolic fixing two cusps. It immediately follows that  $\stab_{\G} M$ 
is infinite in the first two cases, and $|\stab_{\ov{\G}} M|=1$ in the last case. 
\end{proof}
Note that formula~\eqref{TFF} makes sense for an arbitrary
Fuchsian group of the first kind with cusps and an arbitrary double coset $\SS$, 
and we indeed expect it to hold in this level of generality. We therefore 
state the next theorem so that formula~\eqref{TFS3} below holds 
whenever~\eqref{TFF} does, under a mild assumption on the double coset~$\SS$. 
\begin{theorem}\label{TSS1} Let $\G$ be a Fuchsian subgroup of the 
first kind with cusps, and let $\SS\subset \wg$ be a double coset 
satisfying Assumption~\ref{ass4}.  
Then the trace formula~\eqref{TFF} is equivalent to   
 \be\label{TFS3} \begin{split}
 \tr([\SS], S_k(\G,\chi)+S_k^c(\G,\chi))=
 \sum_{\substack{\X,\, \DD(\X)\le 0}}p_{k-2}(\tr M_\X, \det M_\X) \;\wc(M_\X)\;\e_\G(\X) \\
 -  \sum_{a,d>0}\min(a,d)^{k-1}\Phi_{\G,\SS}^\chi (a,d)
  +2 \delta_{k,2}\delta_{\chi,{\bf 1}}\; \sum_{\ss\in\G\backslash\SS} \wc(\ss)\;,
 \end{split} \ee
where the sum is over $\G$-conjugacy classes $\X$ contained in $\ov{\SS}$, 
and $\Phi_{\G,\SS}^\chi$ is defined in~\eqref{1.8}. 
 
In particular, Theorem~\ref{TSS} holds under the assumptions of Theorem \ref{T2}.
 \end{theorem}

Just like the trace formula~\eqref{TFF} is equivalent to that in Theorem~\ref{T2},
formula~\eqref{TFS3} is equivalent to that in Theorem~\ref{TSS} 
(see~\cite[Sec. 4]{P} for the details). Note that by Remark~\ref{r2}, 
we can scale~$\SS$ so that the sum over $a,d$ is over \emph{integers} $a,d>0$ 
with $ad=\det M$, for some~$M\in\SS$.

\comment{
 The sum over $a,d$ in \eqref{TFS2} can be written more intrinsically as follows: 
  \be
  \sum_{a,d>0}\min(a,d)^{k-1}\Phi_{\G,\SS}^\chi (a,d)=
  \sum_{\aa\in C(\G,\chi)}\sum_{ \ss\in\G_\aa\backslash \SS_\aa/\G_\aa }
  \frac{ \min(|\lambda_\ss|,|\lambda_\ss'| )^{k-1} }{(|\lambda_\ss|,|\lambda_\ss'|) }  \sgn(\ss)^k \wc(\ss)\;,\ee
where $\lambda_\ss, \lambda_\ss'$ are the eigenvalues of $\ss$, and $(a,d)$ is 
defined in Theorem~\ref{T5.1}. 
}

\begin{proof}
Let $\tr_{>0}(\G,\chi,\SS,k)$ be the sum in~\eqref{TFF} over the conjugacy 
classes $\X\subset \ov{\SS}$ with~$\DD(\X)>0$. Only the hyperbolic classes 
$\X$ with representatives $M_\X\in\SS$ fixing two (distinct) cusps of~$\G$ 
contribute to the sum, and $\e_\G(\X)=1$ for these classes, by Lemma~\ref{L4.4}. 
Let $H_{\G,\SS}(a,d)\subset \SS$ be a system of representatives for these 
conjugacy classes that have eigenvalues $a,d$ or $-a,-d$. Since 
$p_{k-2}(a+d,ad)=\frac{d^{k-1}-a^{k-1}}{d-a}$, we obtain (recall~$\sgn(\ss)$ 
is the sign of the  eigenvalues of $\ss$):
  \[ \begin{aligned}
  \tr_{>0}(\G,\chi,\SS,k)&=\sum_{d>a>0}\frac{d^{k-1}-a^{k-1}}{d-a} \sum_{\ss\in H_{\G,\SS}(a,d) } \sgn(\ss)^k\wc(\ss)  \\
  &=\sum_{\substack{d>a>0}} (d^{k-1}-a^{k-1})\Phi_{\G,\SS}^\chi(a,d)\\
  &=\tr([\SS], E_k(\G,\chi))-\sum_{a,d>0}\min(a,d)^{k-1}\Phi_{\G,\SS}^\chi(a,d)+
  \delta_{k,2}\delta_{\chi,{\bf 1}}\!\! \sum_{\ss\in\G\backslash\SS}\!\! \wc(\ss)\;,
  \end{aligned} \]
where the second equality follows from part (b), and the third from part (a) of 
Theorem \ref{T5.1}, using also the symmetry of $\Phi_{\G,\SS}^\chi$. The
equivalence of the trace formulas~\eqref{TFF} and ~\eqref{TFS2} is now 
clear.  
\end{proof}
\begin{remark}\label{r4.6}
The sum over conjugacy classes $\X$ with $\DD(\X)=0$ in~\eqref{TFS2}
contains scalar classes only, by the definition of $\e_\G(\X)$, so it equals
 \[\frac{|\G\backslash\H|}{2\pi}\sum_{\lambda} (k-1)\lambda^{k-2}  \wc( 
\lambda I), \]
 where the sum is over $\lambda$ with $\lambda I\in\SS$ and $\lambda>0$ if 
$-1\in \G$. \vspace{2mm}
\end{remark}

\subsection{Another formula for the cuspidal sum~$\Phi_{\G,\SS}^\chi$}\label{sec4}

Assume now that $\G$ is a finite index subgroup of $\G_1=\SL_2(\Z)$. 
To compute explicitly the function~$\Phi_{\G,\SS}^\chi$ appearing in 
Theorems~\ref{T5.1} and~\ref{TSS}, it is convenient to parametrize 
the cusps of~$\G$ by the space of double cosets $\G\backslash\G_1/\G_{1\infty}$, 
where $\G_{1\infty}$ denotes the stabilizer of the cusp~$\infty$ in $\G_1$. 

Let $\chi$ be a character of $\G$, $k\ge 2$, and assume that $\chi(-1)=(-1)^k$ 
if $-1\in \G$. Let $R(\G)\subset \G_1$ be a system of representatives for \emph{the cusp space} 
$\G\backslash\G_1/\G_{1\infty}$, which is in bijection with a set of 
representatives $C(\G)$ for $\G$-equivalence classes of cusps of~$\G$ 
by $C\mapsto \aa=C\infty$. The set $C(\G,\chi)$ introduced in~\S\ref{s4.1.1} 
is then in bijection with a set 
\be \label{39}
  R(\G,\chi)=\big\{C\in R(\G)\,|\, \chi(\e CT^jC^{-1})=\e^k \text{ if } 
  \e CT^jC^{-1}\in \G \text{ for some } \e\in \{\pm 1\}\big\} \;.\ee     
For $C\in R(\G,\chi)$, let $\om(C)$ be the smallest positive integer 
such that $CT^{\om(C)}C^{-1}\in \pm\G $.
\footnote{That is, $\om(C)$ is the width of the cusp $C$ if 
$CT^{\om(C)}C^{-1}\in\G $, or it is half the width if $-1\notin\G$ 
and $CT^{\om(C)}C^{-1}\in-\G $.} 
\begin{lemma} \label{p6.1} Let $\SS\subset \M$ be a double coset satifying Assumption~\ref{eq_star}. 
The function~$\Phi_{\G,\SS}^{\chi}$ introduced in~\eqref{1.8} is given by
  \be\label{42} \Phi_{\G,\SS}^{\chi}(a,d)=\frac{1}{(a,d)}\sum_{C\in R(\G,\chi)} 
  \sum_{\substack{M\in \la T^{\om(C)}\ra\backslash M_{a,d}^\infty/\la T^{\om(C)}\ra \\ \pm CM C^{-1}\in\SS }} (\pm 1)^k\wc(\pm CM C^{-1}) \;,\ee
where $M_{a,d}^\infty=\{\sm ab0d\in \M \}$. 
\end{lemma}
\begin{proof}
For $C\in R(\G,\chi)$, the stabilizer $\G_\aa$ of the cusp $\aa=C\infty$ 
is generated by  $\g_\aa=\pm CT^{\om(C)}C^{-1} \in \G$, and as scaling 
matrix we can take 
\[ C_\aa=\sm {\om(C)^{-1/2}}00{\om(C)^{1/2}}C^{-1}\;. \]
The bijection  $R(\G,\chi)\simeq C(\G,\chi)$ given by $C\mapsto \aa=C\infty$, 
then yields a bijection
\[ \{M\in \la T^{\om(C)}\ra\backslash M_{a,d}^\infty/\la T^{\om(C)}\ra \;:\; \pm CMC^{-1}\in\SS\} 
  \longrightarrow   \G_\aa\backslash\SS_\aa(a,d)/\G_\aa\;, \]
given by $M\mapsto \pm CMC^{-1}\in\SS$, where the sign can be chosen 
positive if $-1\in\G$, and only one choice is possible if $-1\notin \G$ 
(since $\SS\cap(-\SS)=\emptyset$ by Assumption~\ref{eq_star}). We conclude that  
the right hand sides of \eqref{42} and \eqref{1.8} are equal term by term.
\end{proof}

\section{Explicit trace formulas for \texorpdfstring{$\G_0(N)$}{Gamma0[N]}} \label{sec5}

We now specialize $\G=\G_0(N)$ and we compute the functions $\cc_{\G,\SS}^\chi$, 
$\Phi_{\G,\SS}^\chi$ in~\eqref{1.10},~\eqref{1.8} to prove Theorems~\ref{T3} 
and~\ref{T4} in the introduction. We use the formula for 
$\Phi_{\G,\SS}^\chi$ given in~\eqref{42}. In Section~\ref{sec5.1} we consider the usual Hecke operators 
for~$S_k(\G,\chi)$, while in Section~\ref{sec5.2} we consider a composition 
of Hecke and Atkin-Lehner operators on $S_k(\G)$. In Section~\ref{sec4.3} we 
prove Corollary~\ref{C2}.

\subsection{Trace of Hecke operators on 
\texorpdfstring{$\G_0(N)$}{Gamma0[N]} with  Nebentypus}\label{sec5.1}
We take $\G=\Gamma_0(N)$, $\SS=\DD_n$ as in \eqref{delta}, and $\chi$, 
$\wc$ defined there. If $\chi(-1)=(-1)^k$, we have 
$\cc_{\G,\SS}^\chi(M)=\cc_{N,\chi}(M)$, defined in~\eqref{6}. The function 
$\cc_{N,\chi}(M)$ was computed explicitly by Oesterl\'e~\cite[Eq.(35)]{O77},
and it satisfies Assumption~\ref{h}. We sketch the proof in the next lemma, 
since we will use it later. Recall that for $u|N$, $u^2|t^2-4n$, 
in the introduction we have defined  the set
  \[ S_{N}(u,t,n)=\{\alpha\in (\Z/N\Z)^\times: \ \alpha^2-t\alpha +n\equiv 0 
  \pmod{Nu} \}.\] 
\begin{lemma}[Oesterl\'e] \label{L9}For $M=\sm ABCD \in \M_n$, let $t=\tr(M)$, and $u=(G,N)$ where $G$ is the content
of the quadratic form $Q_M=[C,D-A,-B]$. Then 
$$\cc_{N,\chi}(M)=B_{N,\chi}(u, t, n):=\frac{\varphi_1(N)}{\varphi_1(N/u)}\sum_{\alpha\in
S_{N}(u,t,n)} \chi(\alpha).
$$ 
\end{lemma}
\begin{proof}
For $X= \sm abcd \in \G_1$, we have $XMX^{-1}=\sm \a{-Q_M(-b,a)}{Q_M(-d,c)}{t-\a} $, where
$\alpha$ satisfies 
\begin{equation}\label{20.2}
\alpha^2-t\alpha+n = Q_M(-d,c)Q_M(-b,a)\;.
\end{equation}
The condition $XMX^{-1}\in \Delta_n$ is equivalent to $Q_M(d,-c)\equiv 0 \pmod N$ and
$(\alpha,N)=1$, so we have $\alpha\in S_N(u, t,n)$, where we set $u=(G,N)$. Moreover, 
$\alpha$ satisfies $M\big(\begin{smallmatrix}
-d\\c\end{smallmatrix} \big) \equiv \alpha \big(\begin{smallmatrix}
-d\\c\end{smallmatrix} \big) \pmod{N}$, that is
\begin{equation}\label{20.1}
\alpha d \equiv dA-cB \!\!\!\!\!\pmod{N}, \quad \alpha c \equiv cD-dC \!\!\!\!\!\pmod{N}.
\end{equation}
which determines its class mod $N$ uniquely depending only on the point $(c:d)\in
\PP^1(\Z/N\Z)$, namely on the coset of $X$ in $\G_0(N)\backslash\G_1$. Set
$\Sc_N(M)=\{X\in \G_0(N)\backslash\G_1: XMX^{-1}\in \Delta_n \}$.
We have therefore a well-defined map 
\[ 
\tau: \Sc_N(M) \longrightarrow S_N(u, t,n), \quad X\mapsto a_{XMX^{-1}}\;,
\]
and one can show that $|\tau^{-1}(\alpha)|=\varphi_1(N)/\varphi_1(N/u)$ 
independent of $\alpha\in S_N(u, t,n)$, finishing the proof. We refer to~\cite{O77}
for the details. 
\end{proof}
To compute the function $\Phi_{N,\chi}:=\Phi_{\G_0(N),\DD_n}^\chi$ in 
Theorem~\ref{TSS}, we use Lemma~\ref{p6.1}.

\begin{lemma} \label{Ta2} 
Let $k\ge 2$, $N\ge 1$, and $\chi$ a character of conductor $c_\chi|N$ 
such that $\chi(-1)=(-1)^k$. We have 
$$\Phi_{N,\chi}(a,d)=\sum_{\substack{N=rs\\(r,s)|(N/c(\chi),a-d )}}
\varphi((r,s))\chi(\alpha_{r,s}^{a,d})
$$
where $\alpha=\alpha_{r,s}^{a,d}$ is the unique solution mod $\frac{rs}{(r,s)}$ of
$\alpha\equiv a\!\! \pmod{r}$, $\alpha\equiv d\!\! \pmod{s}$. 
\end{lemma}

\begin{proof} Since $-1\in\G$, formula \eqref{42} gives
\[\Phi_{N,\chi}(a,d)=\frac{1}{(a,d)}\sum_{C\in R(\G,\chi)} 
  \sum_{\substack{M\in \la T^{\om(C)}\ra\backslash M_{a,d}^\infty/\la T^{\om(C)}\ra\\CM C^{-1}\in\DD }} \chi(a_{CM C^{-1}})\;.
\]
We choose the set of representatives $R(\G)$ for $\G\backslash\G_1/\G_{1\infty}$ to consist of matrices 
$C=\sm p*rq$ with $N=rs$, and $q$
running through a set $\Sc_r$ with $|\Sc_r|=\varphi((r,s))$. The condition $\chi(CT^jC^{-1})=1$
whenever $CT^jC^{-1}\in \G$ in~\eqref{39} is the same as $\chi(1+prj)=1$ if $N|r^2 j$, namely if
$\frac{N}{(r,s)}|rj$. This can happen if and only if $\chi$ has conductor $c_\chi|
\frac{N}{(r,s)}$, hence $R(\G,\chi)$ consists of those $C\in R(\G)$ as above with
$(r,s)|(N/c_\chi)$. The width of the cusp $C\infty$ is $\om(C)=s/(r,s)$. 

Let $\alpha= a_{CM C^{-1}}$ for a fixed $C\in R(\G,\chi)$ as above, 
and set $M=\sm ab0d$ with $b$ running through the residues modulo $\om(C)g$, 
where $g=(a,d)$. By \eqref{20.1}, the condition $CM C^{-1}\in \Delta$ 
is equivalent to:
  \[ \alpha r\equiv dr \pmod{N}, \quad \alpha q\equiv aq-br \pmod{N}, \quad (\alpha,N)=1\;. \]
Therefore $\alpha \equiv d \pmod{s}$, $\alpha\equiv a
\pmod{r}$, which determines $\alpha$ uniquely modulo $rs/(r,s)$, thus it determines
$\chi(\alpha)$ uniquely since  $c_\chi| \frac{N}{(r,s)}$. We also have $(r,s)|(a-d)$
and since $(s,g)=1$, there are $g$ solutions $b\pmod{gs/(r,s)}$ of the congruence 
$b\equiv q(a-\alpha)/r \pmod s$, independent of $q\in \Sc_r$. We conclude 
  \be\label{31} \Phi_{N,\chi}(a,d)=\sum_{N=rs} \chi(\alpha)\cdot |\Sc_r| \ee
with $\alpha=\alpha_{r,s}^{a,d}$, $|\Sc_r|=\varphi((r,s))$, and $(r,s)|(a-d),
(r,s)|N/c_{\chi}$ in the summation range.  
\end{proof}

\subsection{Trace of Atkin-Lehner and Hecke operators for \texorpdfstring{$\G_0(N)$}{Gamma0[N]}}\label{sec5.2}
Let $\G=\G_0(N)$ and $\chi={\bf 1}$ the trivial character.  
For $N=\ell \ell'$ with $(\ell,\ell')=1$ consider the double coset 
$\Theta_\ell=\G w_\ell \G$, where $w_\ell=\left(\begin{smallmatrix} \ell x & y \\
Nz & \ell t \end{smallmatrix}\right)\in \M_\ell$ with $x,y,z,t\in \Z$. 
The coset space $\G\backslash \Theta_\ell$ consists of one element $\G w_\ell$, 
and the Atkin-Lehner involution $W_\ell$ on $M_k(\G)$ is given by
$W_\ell=\ell^{-w/2}[\Theta_\ell]$ with $w=k-2$. We
let $n\ge 1$ be arbitrary and consider the composition of Hecke and Atkin-Lehner operators
$T_n\circ W_\ell =\frac{1}{\ell^{w/2}}[ \Delta_n \Theta_\ell]$. We have 
\begin{equation}\label{5.1}
\Delta_n\Theta_\ell =\Big\{\left(\begin{smallmatrix} a & b \\
c & d \end{smallmatrix}\right) \in \M_{\ell n} :  N|c,\  \ell|\tr(M), \ \ell|a, \ (a,
\ell')=1, \ (b,\ell)=1 \Big\}.
\end{equation}
The double coset $\Theta_\ell\Delta_n$  is characterized by the same conditions, except
that $(b,\ell)=1$ is replaced by $(c/N, \ell)=1$. If $(n,\ell)=1$, the last two conditions
are empty, so the double cosets $\Delta_n$, $\Theta_\ell$ commute, but that is not the
case when $(n,\ell)>1$, when the double cosets, and the corresponding operators, do not commute. If
$(n,\ell)>1$ the double coset  $\Theta_\ell\Delta_n$ does not satisfy 
Assumption~\ref{eq_star}.  
\begin{lemma}\label{L8}
The double coset $\Delta_n\Theta_\ell$ satisfies Assumption~\ref{eq_star}.
\end{lemma}
\begin{proof}
 Let $\SS=\Delta_n\Theta_\ell$. We have to show that if
$\gamma=\big(\begin{smallmatrix} x & y \\ z & t \end{smallmatrix}\big) \in \G_1$ and
$\sigma=\big(\begin{smallmatrix} a & b \\ c & d \end{smallmatrix}\big)\in \SS$ such
that $\gamma\sigma\in \SS$, then we have $\gamma\in \G_0(N)$. Indeed, we have
$\ell\ell'|az$ and $\ell|bz$ and since $(a, \ell')=1$, $(b,\ell)=1$, we conclude that
$\ell\ell'|z$, so  $\gamma\in \G_0(N)$.
\end{proof} 

For $\G=\G_0(N)$ and $\SS=\DD_n\Theta_\ell$, assuming $k\ge 2$ is even we write 
  \be \label{6.5}\cc_{\G,\SS}^{\bf 1}(M)=\cc_{N,\ell}(M):=\#\{A\in\G_0(N)\backslash\G_1\;:\; AMA^{-1}\in \DD_n\Theta_\ell\}\;. \ee
In the next lemma we compute $\cc_{N,\ell}$ and see that it satisfies Assumption~\ref{h}.
\begin{lemma} \label{L6} For $M\in \M_{\ell n}$, set $t=\tr(M)$,
and let $G$ be the content of the associated quadratic form $Q_M$.
Then the coefficient $\cc_{N,\ell}(M)$ vanishes if
$\ell\nmid t$ and for $\ell|t$ it is given by 
\[
\cc_{N,\ell}(M) =\delta_{(\ell,G),1}\cdot\cc_{\ell',{\bf 1}}(M)=
\sum_{\substack{u|(\ell,G)\\u'|(\ell',G)}} C_{\ell'}(u',t,\ell n)\mu(u)  \;,
\]
where $\cc_{\ell',{\bf 1}}(M)$ is defined in in Lemma \ref{L9} and  
$C_{\ell'}(u,t,n):=C_{\ell',{\bf 1}}(u,t,n)$, for ${\bf 1}$ the 
trivial character modulo $\ell'$.
\end{lemma}  
\noindent The coefficients $C_{\ell'}(u,t,n)$ are computed explicitly in Lemma~\ref{L4}.
\begin{proof} Let $gMg^{-1}=\sm {\alpha}{\beta}{\gamma}{\delta}$.
By \eqref{5.1} and \eqref{6.5}  we have
  \[\cc_{N,\ell}(M)=\big|\{g\in \G\backslash\G_1 : \ell | \tr(M), 
  \ell|\alpha, \ell\ell'|\gamma , (\alpha, \ell')=1, (\beta, \ell)=1\}\big|. \]
Assuming from now on $\ell|\tr(M)$ (otherwise the previous set is empty), let
$M=\sm ABCD$, and $g=\sm abcd \in \G_1$. The coset
of $g$ in $\G\backslash\G_1$ is identified with the point $(c:d)\in \PP^1
(\Z/\ell \ell'\Z)$. By \eqref{20.1}, the conditions $\ell|\g$, $\ell|\alpha$ are equivalent to
\begin{equation}\label{20}
dA\equiv cB \pmod{\ell}, \quad dC\equiv cD \pmod{\ell}.
\end{equation}
Since $\beta=-Q_M(-b,a)$,  the condition $(\beta, \ell)=1$ implies
$(G,\ell)=1$. Conversely, if $p|(\beta, \ell)$ it follows that $p$ divides all
entries of $gMg^{-1}$, so $p|G$. Therefore the
condition $(\beta, \ell)=1$ is equivalent to $(G,\ell)=1$. 

Since $(\ell, \ell')=1$, the Chinese remainder theorem gives 
\begin{equation}\label{21}
\cc_{N,\ell}(M)=\cc_{\ell',\bf{1}}(M) \cdot N_\ell(M)
\end{equation}
with $\cc_{\ell',\bf{1}}(M)$ given by \eqref{16} for the trivial character ${\bf
1}$ mod $\ell'$, and $N_\ell(M)$ equals 0 if $(G,\ell)>1$, while $N_\ell(M)$ denotes the
number of solutions $(c:d)\in \PP^1(\Z/\ell\Z)$ of the system \eqref{20} if 
$(G,\ell)=1$.  As $AD-BC=n\ell$, we have $(G,\ell)=1$ if and only if $(A,B,C,D,\ell)=1$, and 
in this case $N_\ell(M)=1$. This proves the first equality, and the second follows 
by writing $\dd_{(\ell,G),1}=\sum_{u|(\ell,G)} \mu(u).$
\end{proof}

\begin{lemma} \label{L4} 
For any $N\ge 1$ and for $u|N, u^2|t^2-4n$ we have
\[
C_{N}(u,t,n)= |S_N(t,n)| \cdot C_N(u,t^2-4n)
\]
where $S_N(t,n)= \{\alpha\in (\Z/N\Z)^\times: \ \alpha^2-t\alpha +n\equiv 0 \pmod{N}\}$,
and the coefficients $C_N(u,D)$, defined for $u|N, u^2|D$, are
multiplicative in $(N,u)$, namely
\[
C_N(u,D)=\prod_{p|N} C_{p^{\nu_p(N)}}(p^{\nu_p(u)},D).
\]
If $N=p^a$ with $p$ prime and $a\ge 1$  we have  $C_{N}(p^0,D)=1, \ 
C_{N}(p^a,D)=p^{\lceil\frac{a}{2} \rceil},$ 
and setting $b=\nu_p(D)$ (with $b=\infty$ if $D=0$), for $0< i<a$ we have: if $p$ is
odd then 
\[ 
C_N(p^i,D)=\begin{cases}
  p^{\lceil\frac{i}{2} \rceil}- p^{\lceil\frac{i}{2} \rceil-1}& \text{ if }  1\le i\le
	 b-a,  i\equiv a\!\!\!\! \pmod{2}\\
  -p^{\lceil\frac{i}{2} \rceil-1} & \text{ if }  i=b-a+1, i\equiv a\!\!\!\! \pmod{2}\\
 p^{\lfloor\frac{i}{2} \rfloor}\Big(\tfrac{D/p^b}{p}\Big) & \text{ if }   i=b-a+1,
i\not\equiv a\!\!\!\!
\pmod{2}\\
  0& \text{ otherwise},
             \end{cases}
\] 
while if $p=2$ then
\[
C_N(2^i,D)=\begin{cases}
  2^{\lceil\frac{i}{2} \rceil-1}& \text{ if }  1\le i\le b-a-2,i\equiv a\!\!\!\! \pmod{2}
\\
 -2^{\lceil\frac{i}{2} \rceil-1}& \text{ if }  i=b-a-1,i\equiv a\!\!\!\! \pmod{2} \\
  2^{\lceil\frac{i}{2} \rceil-1}\epsilon_4(D/2^b)& \text{ if }  i=b-a, i\equiv a\!\!\!\!
\pmod{2}\\
  2^{\lfloor\frac{i}{2} \rfloor}\Big(\tfrac{D/2^b}{2}\Big) & \text{ if }  i=b-a+1, i\not\equiv 
a\!\!\!\!\pmod{2} \text{ and }D/2^b\equiv 1 \!\!\!\!\pmod 4 \\                   
  0& \text{ otherwise},
             \end{cases}
\]
where $\Big(\tfrac{\bullet}{p}\Big)$ denotes the Legendre symbol, and $\epsilon_4$ is the notrivial 
quadratic character mod 4. 
\end{lemma}
\begin{proof} If $S_N(t,n)=\emptyset$, then $S_{N}(u,t,n)=\emptyset$ for all $u|N$, so the
constants $C_{N}(u,t,n)$ are all 0. 
Since $|S_{N}(u,t,n)|=|S_{N}(u,-t,n)|$, we can write
\[C_{N}(u,t,n)=|S_N(t,n)| \cdot C_N(u,D)
\]
for some function $C_N(u,D)$. The function
$C_N(u,D)$ is
multiplicative in $N$, as $C_{N,{\bf 1}}(u,t,n)$ and $|S_N(t,n)|$ are multiplicative, and, by \eqref{16},
if $N=p^a$ it equals 
\[C_{N}(p^i,D)=\frac{\varphi_1(p^a)/\varphi_1(p^{a-i})\cdot | S_{	
N}(p^i,t,n)|-
\varphi_1(p^a)/\varphi_1(p^{a-i+1})\cdot |S_{N}(p^{i-1},t,n)|}{|S_N(t,n)|},
\]
with the second term in the numerator missing if $i=0$. The cardinality of the set
$S_{p^a}(p^i,t,n)$ is straightforward to compute, leading to the formulas above. 
\end{proof}

To finish the proof of Theorem \ref{T4}, we compute the function 
$\Phi_{N,\ell}:=\Phi_{\G,\DD_n\Theta_\ell}^{\chi=\bf{1}}$ in Theorem~\ref{p6.1}.

\begin{lemma} \label{Ta1}
We have $\Phi_{N,\ell}(a,d)=0$ unless $n\ell=ad$ and $\ell|a+d$, when 
  \[\Phi_{N,\ell}(a,d)=\frac{\varphi(\ell)}{\ell}\sum_{\substack{\ell'=rs,\ (r,s)|a-d\\   
  (r,a)=1, (s,d)=1}} \varphi((r,s))\;.\] 
\end{lemma}
\begin{proof}Letting $M_{b}=\sm ab0d$ and $g=\gcd(a,d)$, by Lemma~\ref{p6.1} we have
  \[ \Phi_{N,\ell}(a,d)=\frac 1{g} \sum_{C\in R(\G)}  \#\{b\!\!\!\! \pmod{g\om(C)} \;:\; 
  CM_b C^{-1}\in\DD_n\Theta_\ell \},\]
where $R(\G)\subset \G_1$ is the set of representatives for the cusp
space $\G\backslash\G_1/\G_{1\infty}$ chosen in the proof of Lemma~\ref{Ta2}.
Counting elements in the sets above can now be done as in the proof of
Lemma~\ref{Ta2}. We omit the proof, which can be found in our
arXiv preprint 1408.4998v2.
\end{proof}
\subsection{Proof of Corollary~\ref{C2}}\label{sec4.3}
The function $\Psi_4(a,d)$ is easily computable, so it remains to compute 
the sum over $t$ in Theorem~\ref{T5}. We write it as:
\[
\sideset{}{'}\sum_{0\le t\le\sqrt{4n}}p_{k-2}(t,n)\cdot
\sum_{u|4} H\left(\tfrac{4n-t^2}{u^2}\right) (D_4(u,t,n)+(-1)^k D_4(u,-t,n))
\]
with the prime indicating the the term with $t=0$ has coefficient 1/2. The sum
is restricted to those $t$ with $4|\pm t-n-1$, as otherwise $D_4(u,\pm t,n)=0$.
For $n$ even this means that $t$ is odd so only $u=1$ is present in the sum over 
$u$, obtaining the formula in part b). For $n$ odd, it follows that $t=2s$ is even 
and we have a few cases depending on whether $\nu_2(2s-n-1)$ is 2, 3 or $\ge 4$. 
We leave the details to the reader, noting only that in the case 
$k$ even and $n$ odd we also use the class number relations in the 
following table, valid for $D\ge 0$:\vspace{2mm}
\setlength{\tabcolsep}{12pt}
\begin{center}
\begin{tabular}{c|c|c}
$D\equiv 3 \pmod 8$ &  $D\equiv 7 \pmod 8$ & $D\equiv 0 \pmod 4$ \\
$H(4D)=4H(D)$ &  $H(4D)=2H(D)$ & $H(4D)=3H(D)-2H(D/4)$ \\
\end{tabular}
\end{center}

\end{document}